\documentclass[a4paper]{article}
\usepackage{fullpage}
\usepackage{amsmath}
\usepackage{amssymb}
\usepackage{amsthm}
\usepackage{graphicx}
\usepackage{bbm}
\usepackage{enumerate}
\usepackage{microtype}
\usepackage{tocbibind}
\usepackage{xcolor}
\usepackage[none]{hyphenat}
\usepackage[colorlinks=true,urlcolor=blue,linkcolor=blue,plainpages=false,pdfpagelabels]{hyperref}
\allowdisplaybreaks

\newtheorem{theorem}{Theorem}[section]
\newtheorem{proposition}[theorem]{Proposition}

\newtheorem{corollary}[theorem]{Corollary}
\newtheorem{remark}[theorem]{Remark}

\newtheorem{definition}[theorem]{Definition}

\newcommand{\N}{\mathbb{N}}
\newcommand{\Q}{\mathbb{Q}}
\newcommand{\R}{\mathbb{R}}
\newcommand{\C}{\mathbb{C}}
\newcommand{\T}{\mathbb{T}}
\newcommand{\eps}{\varepsilon}

\newcommand{\uu}[1]{\underline{#1}}
\newcommand{\wt}[1]{\widetilde{#1}}

\newcommand{\tp}{\mathrm{tp}}

\DeclareMathOperator{\tr}{tr}
\newcommand{\Tt}{\mathcal{T}}
\newcommand{\St}{\mathcal{S}}
\newcommand{\wh}[1]{\widehat{#1}}
\DeclareMathOperator{\da}{-}
\DeclareMathOperator{\vN}{vN}
\newcommand{\ff}{\mathrm{f}}
\newcommand{\fff}{\mathrm{II_1}}
\title{A proof-theoretic metatheorem for tracial von Neumann algebras}
\author{Liviu P\u aunescu${}^{b}$ and Andrei Sipo\c s${}^{a,b,c}$\\[2mm]
\footnotesize ${}^a$Research Center for Logic, Optimization and Security (LOS), Department of Computer Science,\\
\footnotesize Faculty of Mathematics and Computer Science, University of Bucharest,\\
\footnotesize Academiei 14, 010014 Bucharest, Romania\\[1mm]
\footnotesize ${}^b$Simion Stoilow Institute of Mathematics of the Romanian Academy,\\
\footnotesize Calea Grivi\c tei 21, 010702 Bucharest, Romania\\[1mm]
\footnotesize ${}^c$Institute for Logic and Data Science,\\
\footnotesize Popa Tatu 18, 010805 Bucharest, Romania\\[2mm]
\footnotesize E-mails: liviu.paunescu@imar.ro, andrei.sipos@fmi.unibuc.ro\\
}
\date{}

\begin{document}

\maketitle

\begin{abstract}
We adapt a continuous logic axiomatization of tracial von Neumann algebras due to Farah, Hart and Sherman in order to prove a metatheorem for this class of structures in the style of proof mining, a research program that aims to obtain the hidden computational content of ordinary mathematical proofs using tools from proof theory.


\end{abstract}

\section{Introduction}

One manner in which mathematical logic may have something to say about structures prevalent in e.g. functional analysis arises from the research area of {\it proof mining}, which aims to apply proof-theoretic methods and tools to ordinary mathematical proofs in order to uncover some additional information which is most often of a quantitative (computational) nature. The program was first suggested by Georg Kreisel in the 1950s \cite{Kre51,Kre52} under the name `unwinding of proofs', and given maturity, under its current name (due to Dana Scott), by the school of Ulrich Kohlenbach, starting in the 1990s -- for detailed references, see e.g. \cite{Koh08,Koh19}.

The theoretical basis of proof mining consists of several `general logical metatheorems', first introduced in \cite{Koh05, GerKoh08}, that are vastly expanded versions (in the sense that they are able to deal with abstract structures in addition to the usual higher-order arithmetical reasoning) of functional interpretations (most notably G\"odel's {\it Dialectica} interpretation \cite{God58}), building up on Kohlenbach's 1990s work on {\it monotone functional interpretations} \cite{Koh96}. From time to time, advances in proof mining required various extensions of the metatheorems whose proofs could require intricate arguments, see e.g. \cite{Leu06, Leu14, GunKoh16, Sip17, KohNic17, Sip19, Pis22} (as well as the recasting \cite{EngFer20} of the metatheorems in the framework of the bounded functional intepretation due to Ferreira and Oliva \cite{FerOli05}).

Such an extension was presented in the paper \cite{GunKoh16}, and its motivation was as follows: since there exists also a well-established model-theoretic link between mathematical logic and (functional) analysis, namely through the so-called {\it logics for metric structures}, including e.g. continuous logic and positive-bounded logic, in which classes of structures may be given appropriate axiomatizations, the question arises whether these axiomatizations could be translated into proof-theoretic logical systems that could admit logical metatheorems. Such a translation for positive-bounded formulas was presented and shown to work by G\"unzel and Kohlenbach in the paper \cite{GunKoh16} mentioned above. This work was continued in \cite{Sip19} where a positive-bounded axiomatization and a corresponding logical metatheorem were given for the class of $L^p$ Banach spaces. In particular, the present paper could be construed as a loose sequel to \cite{GunKoh16} and \cite{Sip19}.

Specifically, our goal is to state and prove a general logical metatheorem for tracial (complex) von Neumann algebras, a class of structures of great importance in the fields of operator algebras, ergodic theory, geometric group theory (including sofic groups), probability theory etc. The start of our investigations is a model-theoretic axiomatization of it due to Farah, Hart and Sherman \cite{FarHarShe14}; we note that this axiomatization is closely connected to the recent claimed negative solution to the Connes Embedding Problem via the complexity-theoretic statement of $\mathsf{MIP^*=RE}$ \cite{jimip}, namely model theory can provide a `shortcut' (see \cite[Section 7]{Gol21} and \cite[Subsection 3.3]{Gol22}) in showing that $\mathsf{MIP^*=RE}$ implies the failure of the CEP).

Generally, proof-theoretic frameworks were considered to be less restrictive than model-theoretic ones (traditionally, the main advantage was considered to be the ability to deal with weaker forms of extensionality). For example, what we can do in our treatment is to add the algebra norm $\|\cdot\|_\infty$ as an explicit symbol, unlike in \cite{FarHarShe14} where it was present only through the so-called `domains of quantification'. On the other hand, the main difficulty that arose in our work was of a different nature: it was the fact that complex numbers had never been considered before in proof mining, and unlike the norm or distance functions which only take non-negative values, the (real and imaginary) components of the trace may take arbitrary real values, and thus a new representation function $(\cdot)_+$ (see Subsection~\ref{sec:models}) had to be devised.

A particular novelty of the present paper is that we shall use the style of presenting metatheorems originating in the upcoming work \cite{CheLeuSipXX}, this meaning primarily that we use logical systems of proof-theoretic strength around first-order arithmetic, and that the extension to higher types of Kohlenbach's monotone functional interpretation is made syntactically explicit.

Section~\ref{sec:tr} will present all the necessary preliminary facts on tracial von Neumann algebras, in a classical, purely analytical framework. Section~\ref{sec:sys} will show how the theory of von Neumann algebras may be encoded as a logical system (actually a series of such systems), which is also given a standard set-theoretic semantics. Section~\ref{sec:metath} will present the proof of our main theorem, the general logical metatheorem for tracial von Neumann algebras -- Theorem~\ref{metric-cl}. In Section~\ref{sec:conc} we shall suggest avenues for future work, including potential applications.

\section{Tracial von Neumann algebras}\label{sec:tr}

We shall work here over the field of the complex numbers, i.e. by vector spaces (or algebras) we shall mean $\C$-vector spaces (resp. $\C$-algebras).

\begin{definition}
An {\bf algebra} is a vector space $A$, where we denote addition by $+$ and scalar multiplication by $\cdot$, equipped with a multiplication operation, denoted also by $\cdot$, such that $(A,+,\cdot)$ is a ring and for any $\lambda \in\C$ and $x$, $y \in A$, we have that
$$\lambda \cdot (x \cdot y) = (\lambda \cdot x) \cdot y = x \cdot (\lambda \cdot y).$$
\end{definition}

From now on, we shall omit the $\cdot$ sign(s).

\begin{definition}
A {\bf Banach algebra} is an algebra $A$ which is also a Banach space with norm $\|\cdot\|_\infty$, such that for every $x$, $y \in A$ we have that
$$\|x  y\|_\infty \leq \|x\|_\infty \|y\|_\infty.$$
\end{definition}

\begin{definition}
A {\bf $C^*$-algebra} is a Banach algebra $A$ equipped with a map $(\cdot)^* : A \to A$ such that for any $\lambda \in\C$ and $x$, $y \in A$, we have that:
\begin{enumerate}[(i)]
\item $x^{**}=x$;
\item $(x+y)^* = x^*+y^*$;
\item $(xy)^* = y^*x^*$;
\item $(\lambda x)^* = \overline{\lambda} x^*$;
\item $\|x^*x\|_\infty = \|x\|_\infty \|x^*\|_\infty$.
\end{enumerate}
\end{definition}

\begin{remark}
\begin{enumerate}
\item  Let $A$ be a $C^*$-algebra. Then:
\begin{enumerate}
\item we have that $0^*=0$ and $1^*=1$;
\item for any $x \in A$, we have that $\|x^*\|_\infty=\|x\|_\infty$ and thus that $\|xx^*\|_\infty = \|x\|_\infty^2$.
\end{enumerate}
\item When listing all the axioms of $C^*$-algebras (as in, e.g., \cite[p. 485]{FarHarShe14}), one can omit the right distributivity axiom, as one has that for any $x$, $y$, $z \in A$,
\begin{align*}
(x+y)z &= ((x+y)z)^{**} \\
&= (z^*(x+y)^*)^*= (z^*x^* + z^*y^*)^* = (z^*x^*)^* + (z^*y^*)^* = x^{**}z^{**} + y^{**}z^{**} = xz+yz.
\end{align*}
\end{enumerate}
\end{remark}

A {\bf predual} of a $C^*$-algebra $A$ is a pair $(V,i)$ where $V$ is a Banach space and $i$ is an isomorphism between the continuous dual of $V$ and $A$. Morphisms between preduals of a single algebra are defined in the natural way, by diagram commutativity. It is a theorem (due to Sakai \cite{Sak56}) that a $C^*$-algebra has at most one predual up to isomorphism, and if that is the case, we call the algebra a {\bf von Neumann algebra}.

\begin{definition}
A {\bf tracial von Neumann algebra} is a von Neumann algebra $A$ equipped with a $\C$-linear functional $\tr : A \to \C$, called its {\bf trace}, that in addition satisfies the following properties:
\begin{itemize}
\item for every $x \in A$, $\tr(x^*)=\overline{\tr(x)}$;
\item for all $x \in A$, $\tr(x^*x) \geq 0$ (with equality iff $x=0$);
\item $\tr(1)=1$;
\item $\tr$ commutes with limits of increasing nets in $A$;
\item for all $x$, $y \in A$, $\tr(xy)=\tr(yx)$.
\end{itemize}
\end{definition}

Generally, if $A$ is a tracial von Neumann algebra with trace $\tr: A \to \C$, we denote, for any $x \in A$, $\|x\|_2 := \sqrt{\tr(x^*x)}$. We have that $\|\cdot\|_2$ is a norm on $A$ considered as a vector space, and that for every $x$, $y \in A$, $\|xy\|_2 \leq \|x\|_\infty \|y\|_2$. In particular, for every $x \in A$, $\|x\|_2 \leq \|x\|_\infty$.

Examples of tracial von Neumann algebras include finite-dimensional matrix algebras $\mathcal{M}_n(\C)$, with the usual matrix trace, as well as those of the form $L^\infty(\Omega,\mathcal{F},\mu)$, where the trace is given by
$$[f] \mapsto \int_\Omega f d\mu.$$
It is known that a tracial von Neumann algebra is commutative iff it is of this latter form.

We also have the following property.

\begin{proposition}
Let $A$ be a tracial von Neumann algebra and $x$, $y \in A$. Then $|\tr(xy)| \leq \|x\|_2 \|y\|_2$.
\end{proposition}

\begin{proof}
First, we remark that it is enough to assume that $\tr(xy) \in \R_+$, since we could, otherwise, take $\mu \in \C$ with $|\mu|=1$ such that $\mu\tr(xy)=\tr(\mu xy) \in \R_+$ and derive that
$$|\tr(xy)|  =|\mu \tr(xy)|=|\tr(\mu xy)| \leq \|\mu x\|_2\|y\|_2 = \|x\|_2\|y\|_2.$$
With that in mind, we also note that
$$\tr(y^*x^*) = \tr((xy)^*) = \overline{\tr(xy)} = \tr(xy).$$
Now, let $f : \R \to \R$, defined, for any $\lambda \in \R$, by
$$f(\lambda) := \tr((x^*+\lambda y)^*(x^*+\lambda y)) \geq 0.$$
We remark, using our assumption from before, that, for any $\lambda \in \R$,
$$f(\lambda) = \|y\|_2^2 \lambda^2 + 2\tr(xy)\lambda + \|x\|_2^2,$$
so $f$ is a quadratic function whose discriminant is $4\tr^2(xy) - 4\|x\|_2^2\|y\|_2^2$, and, since $f$ is always non-negative, we have that said discriminant must be smaller or equal to $0$, from which we derive our conclusion.
\end{proof}

\begin{corollary}
Let $A$ be a tracial von Neumann algebra and $x \in A$. Then $|\tr(x)| \leq \|x\|_2$.
\end{corollary}

In \cite[p. 485--487]{FarHarShe14}, the authors give a complete axiomatization of tracial von Neumann algebras in (an inessential variant of) continuous logic (the details of which need not concern us here). In additional to (some of) the properties given above, they add a final axiom schema, which contains, for each $n \in \N^*$ and $*$-polynomial $p$, an axiom which roughly states, denoting by $m$ the number
$$\max \{ a \in \R \mid \text{there is a tracial von Neumann algebra $A$ and $x \in A$ with $\|x\| \leq n$ s.t. $\|p(x)\|_\infty=a$}\},$$
that for every element $x$ of the axiomatized structure of norm less than or equal to $n$, $\|p(x)\|_\infty\leq m$. This circumlocution is needed because the framework of continuous logic, as mentioned in the Introduction, one cannot introduce $\|\cdot\|_\infty$ as a distinct symbol.

We first note that one can omit the bounding of $x$ by writing
$$p\left(\frac{x}{\max(\|x\|_\infty,n)}\right).$$
Then, an inspection of their proof of the correctness of their axiomatization shows that one only uses the axiom schema, for any $n \in \N$, for a denumerable list of $*$-polynomials $q_{n,k}$ for which the corresponding $m$ is always $1$. Those polynomials are constructed as follows. We fix an $n \in \N$, and define the function $t_n : [0,n] \to \R$, for any $x \in [0,n]$, by
$$t_n(x) :=
  \begin{cases} 
      \hfill 1, \hfill &  \text{if $x \leq 1$,}\\
      \hfill \frac1{\sqrt{x}}, \hfill & \text{otherwise.}\\
  \end{cases}$$
Then one defines for each $k \in \N$ a polynomial $q_{n,k}$ (which we fix here for the remainder of the paper) such that the associated sequence of polynomial functions $(q_{n,k})_{k \in \N}$ tends to $t_n$ from below (this can be arranged with a suitable modification of the Bernstein construction). All these preparations were needed in order to obtain a recursively enumerable list of axioms.

In addition to the above, the paper \cite{FarHarShe14} gives the following characterizations (see also \cite[Lemma 4.2]{FarHarShe13} for clarification on the form used here).

\begin{proposition}[{cf. \cite[Proposition 3.4]{FarHarShe14}}]\label{factori}
Let $A$ be a tracial von Neumann algebra. Then:
\begin{enumerate}[(a)]
\item $A$ is a tracial von Neumann factor iff for all $x\in A$ with $\|x\|_\infty \leq 1$ and  for all $\eps>0$ there is a $y \in A$ with $\|y\|_\infty \leq 1$ such that
$$\|x - \tr(x)\|_2 \leq \|xy-yx\|_2 +\eps.$$
\item $A$ is a II$_1$ factor iff the above condition holds and, in addition, for every $\eps>0$ there is an $x\in A$ with $\|x\|_\infty \leq 1$ such that
$$\|xx^* - (xx^*)^2\|_2 + |\tr(x^*x) - 1/\pi|  \leq \eps.$$
\end{enumerate}
\end{proposition}

\section{The logical systems}\label{sec:sys}

Let $0$ and $X$ be such that $0 \neq X$. We shall denote by $\T^X$ the set of all finite types constructed from $X$, i.e. the underlying set of the free algebra with a unique binary operation $\to$ generated by the set $\{0,X\}$. For example, $0$, $X$, $(0 \to 0) \to 0$, $X \to 0$, $0 \to (X \to X)$ are all elements of $\T^X$. For each $\rho \in \T^X$, we shall denote by $\wh{\rho}$ the type obtained from $\rho$ by replacing every occurrence of $X$ by $0$. We shall set $1:=0\to 0$.

We call a finite sequence of types $\vec{\rho} = (\rho_1,\ldots,\rho_n) \in (\T^X)^*$ a {\it type-tuple}. We extend the operation $\to$ to type-tuples, progressively, first to an operation $\to : (\T^X)^* \times \T^X \to (\T^X)^*$, recursively, by
$$() \to \rho := \rho, \quad (\vec{\theta}, \tau) \to \rho:= \vec{\theta} \to (\tau \to \rho),$$
then to an operation $\to : (\T^X)^* \times (\T^X)^* \to (\T^X)^*$, again recursively, by
$$\vec{\rho} \to () := (), \quad \vec{\rho} \to (\vec{\theta},\tau) := (\vec{\rho} \to \vec{\theta}, \vec{\rho} \to \tau).$$
If we have a type-tuple $\vec{\rho} = (\rho_1,\ldots,\rho_n) \in (\T^X)^*$, we say that a {\it term-tuple} of type-tuple $\vec{\rho}$ is a tuple of terms $\uu{t}=(t_1,...,t_n)$ such that for each $i$, $t_i$ is a term of type $\rho_i$, denoting this situation by $\tp(\uu{t}) = \vec{\rho}$; we shall also say that a {\it constant} of type-tuple $\vec{\rho}$ (or, by abuse of language, of type $\vec{\rho}$) is a tuple of constants of corresponding types. The application of terms can be extended recursively to term-tuples -- $t()$ will be $t$; $t(\uu{t'},u)$ will be $(t\uu{t'})u$; $()\uu{v}$ will be $()$; and, finally, $(\uu{t},u)\uu{v}$ will be $(\uu{t}\uu{v},u\uu{v})$. (Note that we do not have product types in our systems.)

We shall refer in the sequel to the representation of rational and real numbers in systems used in proof mining like $\mathsf{WE{\da}HA}^\omega$ (see \cite[Section 3.3]{Koh08}), as referred to in \cite[pp. 78--81]{Koh08}. If $j:\N^2 \to \N$ is the Cantor pairing function, then for each $n$, $m \in \N$, if $n$ is even, $j(n,m)$ will encode the rational $n/(2m+2)$, while if $n$ is odd, $j(n,m)$ will encode the rational $- (n+1)/(2m+2)$ -- in particular, for any $n \in \N$, $j(2n,0)$ encodes $n$. In addition, each function $f: \N \to \N$ (i.e. an object of type $1$) is thought of as representing a unique real number denoted by $\alpha_f$ (so one works throughout only with representatives). One then has: (i) $\forall$-formulas representing equality and order relations on $\R$, denoted by $=_\R$ and $\leq_\R$; (ii) higher-order primitive-recursive encodings of common operations on $\R$; (iii) an `embedding' of $\Q$ (and in particular of $\N$) into $\R$ by associating to every code of a rational number the function constantly equal to that code.

In addition, we shall also work with representations of complex numbers, in a largely {\it ad hoc} manner. We say that the `type' $\C$ is simply the type-tuple $(1,1)$, and if $t=(t_1,t_2)$ and $s=(s_1,s_2)$ are term-tuples of type $\C$, we write $t=_\C s$ for
$$(t_1 =_\R s_1) \land (t_2 =_\R s_2).$$
It is clear that one can immediately define in this manner all the usual operations on complex numbers, denoted here by $*_\C$ (conjugation), $|\cdot|_\C$ (absolute value), $+_\C$, $\cdot_\C$ etc.

Now, in order to be able to `do' proof mining on tracial von Neumann algebras, we shall define two `higher-order' logical systems $\Tt_{\vN,i}$ and $\Tt_{\vN,c}$ (to be considered, in a sense, `intuitionistic' and `classical', respectively), and a third denoted by $\Tt_{\vN,m}$. We shall start by defining the system $\Tt_{\vN,i}$.

As in \cite[pp. 388--389, 411--413, 435--436]{Koh08}, the language of $\Tt_{\vN,i}$ is obtained by extending that of $\mathsf{WE{\da}HA}^\omega$ (see again \cite[Section 3.3]{Koh08}) to all the types in $\T^X$, i.e. by having variables of each type, and also by having recursor and combinator constants again for all the types in $\T^X$. We shall in addition have the following constants:
\begin{itemize}
\item $0_X$ of type $X$;
\item $+_X$ of type $X \to (X \to X)$;
\item $-_X$ (unary minus) of type $X \to X$;
\item $\cdot$ (multiplication with scalars) of type $\C \to (X \to X)$;
\item $\cdot_X$ (algebra multiplication) of type $X \to (X \to X)$;
\item $1_X$ of type $X$;
\item $*_X$ of type $X \to X$;
\item $\|\cdot\|_\infty$ of type $X \to 1$;
\item $\|\cdot\|_2$ of type $X \to 1$;
\item $\tr$ (actually a pair of constants) of type $X \to \C$.
\end{itemize}

We allow infix notation, along with ``syntactic sugar'', e.g. the representations of rational, real and complex numbers introduced before, as well as writing $x -_X y$ for $x +_X (-_X y)$.

We shall use the following abbreviations:
\begin{enumerate}[(i)]
\item for all terms $s$, $t$ of type $X$, $s=_X t$ shall mean
$$\| s -_X t \|_\infty =_\R 0_\R,$$
and equality of higher types shall be defined extensionally as in \cite[Section 3.3]{Koh08}.
\item for each type $\rho$ we define a relation $\preceq_\rho$ between terms of type $\rho$, as in \cite[Definition 3.22]{Koh05} and \cite[Definition 17.64]{Koh08}, recursively, as follows:
\begin{itemize}
\item $s \preceq_0 t$ means $s \leq t$;
\item $s \preceq_X t$ means $\|s\|_\infty  \leq_\R \|t\|_\infty$;
\item for each $\rho$, $\tau \in \T^X$, $s \preceq_{\rho \to \tau} t$ means $\forall z^\rho (s(z) \preceq_\tau t(z))$.
\end{itemize}
\item for each type $\rho$ we define a relation of majorization $\lesssim_\rho$ between one term of type $\rho$ and one of type $\wh{\rho}$ (due originally to Howard \cite{How73}; strong majorization, due to Bezem \cite{Bez85}, was for the first time extended to the context of abstract types by Kohlenbach \cite{Koh05} and Gerhardy and Kohlenbach \cite{GerKoh08}; to our knowledge, this and the one in \cite{CheLeuSipXX} are the first instances of Howard's simple majorization in abstract types, which is due to our use of systems of lower strength), recursively, as follows:
\begin{itemize}
\item $s \lesssim_0 t$ means $s \leq t$;
\item $s \lesssim_X t$ means $\|s\|_\infty \leq_\R (t)_\R$;
\item for each $\rho$, $\tau \in \T^X$, $s \lesssim_{\rho \to \tau} t$ means
$$\forall x^\rho \forall y^{\wh{\rho}} (x \lesssim_\rho y \to sx \lesssim_\tau ty) .$$
\end{itemize}
\end{enumerate}

Now, in order to make $\Tt_{\vN,i}$ into a logical system, we add, firstly (again, as in \cite[Chapter 17]{Koh08}), all the axioms and rules of $\mathsf{WE{\da}HA}^\omega + \mathsf{QF{\da}AC}$ (as extended to all abstract types), and, secondly, the following list of purely universal axioms:
\begin{itemize}
\item the vector space axioms;
\item the ring axioms;
\item the algebra axioms;
\item the normed space axioms (the ones from \cite[p. 411]{Koh08} as adapted for complex scalars);
\item the normed algebra axiom;
\item trace axioms:
\begin{enumerate}[(i)]
\item axioms expressing that the trace is a $\C$-linear functional;
\item $\forall x^X \left( \tr(*_X(x)) =_\C *_\C(\tr(x)) \right)$;
\item $\tr(1_X)=_\C 1_\C$;
\item $\forall x^X \forall y^X \left(\tr(x \cdot_X y) =_\C \tr(y \cdot_X x) \right)$;
\end{enumerate}
\item the axiom linking the trace to the $2$-norm (which also implies the property that for all $x$ in the algebra, $\tr(x^*x)$ is a non-negative real number):
$$\forall x^X \left((\|x\|_2 \cdot_\R \|x\|_2,0_\R) =_\C \tr(*_X(x_X) \cdot_X x_X) \right);$$
\item normed space axioms for $\|\cdot\|_2$ (again, see \cite[p. 411]{Koh08});
\item the continuity axiom:
$$\forall x^X \forall y^X\left( \| x \cdot_X y \|_2 \leq_\R \|x\|_\infty \cdot_\R \|y\|_2 \right);$$
\item special extensionality axioms:
\begin{enumerate}[(i)]
\item $\forall x^X \forall y^X \left(\|*_X(x) -_X *_X(y)\|_\infty =_\R \|x-_X y\|_\infty\right)$;
\item $\forall x^X \forall y^X \forall z^X \left(\|x \cdot_X (y -_X z)\|_\infty =_\R \|(x \cdot_X y) -_X (x \cdot_X z)\|_\infty \right)$;
\item $\forall x^X \forall y^X \forall z^X \left(\|  (x -_X y)\cdot_X z\|_\infty =_\R \|(x \cdot_X z) -_X (y \cdot_X z)\|_\infty\right)$;
\item $\forall x^X \left(\|x\|_2 \leq_\R \|x\|_\infty\right)$;
\item $\forall x^X \left(|\tr(x)|_\C \leq_\R \|x\|_2\right)$;
\end{enumerate}
\item polynomial axiom schema: for all $n \in \N^*$ and $k \in \N$,
$$\forall x^X\left( q_{n,k}\left(\frac{x}{\max(\|x\|_\infty,n)}\right) \leq_\R 1_\R \right).$$
\end{itemize}

The system $\Tt_{\vN,c}$ is obtained from the system $\Tt_{\vN,i}$ by adding the law-of-excluded-middle schema $A \lor \neg A$, while the system $\Tt_{\vN,m}$ is obtained from the system  $\Tt_{\vN,c}$ by adding, for each $\rho \in \T^X$ and each constant $c$ added above of type $\rho$, a new constant $m_c$ of type $\wh{\rho}$ together with the axiom $c \lesssim_\rho m_c$. Those latter axioms (may) have a high logical complexity which would pose problems if they were used as inputs to a functional interpretation (a complex way to mitigate this can be found in \cite{EngFer20}), hence we only use them in this general system, which will only serve in an output role and is the only one which we shall interpret semantically in the results to come.

\begin{proposition}
All the new constants are provably extensional in $\Tt_{\vN,c}$.
\end{proposition}

\begin{proof}
The proofs for $+_X$, $-_X$, $\cdot$ and $\|\cdot\|_\infty$ work just as in the case of normed spaces, see \cite[p. 412]{Koh08}.

For the extensionality of $*_X$, i.e.
$$\forall x_1^X \forall x_2^X \left(x_1 =_X x_2 \to *_X(x_1) =_X *_X(x_2) \right),$$
one uses the special extensionality axiom (i).

For the extensionality of $\cdot_X$, i.e.
$$\forall x_1^X \forall x_2^X \forall y_1^X \forall y_2^X\left(x_1 =_X x_2 \land y_1 =_X y_2 \to x_1 \cdot_X y_1 =_X x_2 \cdot y_2 \right),$$
one uses special extensionality axioms (ii) and (iii).

For the extensionality of $\|\cdot\|_2$, i.e.
$$\forall x_1^X \forall x_2^X \left(x_1 =_X x_2 \to \|x_1\|_2 =_\R \|x_2\|_2 \right),$$
one uses the special extensionality axiom (iv).

For the extensionality of $\tr$, i.e.
$$\forall x_1^X \forall x_2^X \left(x_1 =_X x_2 \to \tr(x_1) =_\C \tr(x_2) \right),$$
one uses special extensionality axioms (iv) and (v).
\end{proof}

We define the corresponding theories for tracial von Neumann factors $\Tt_{\vN,i}^\ff$ (and also $\Tt_{\vN,c}^\ff$, $\Tt_{\vN,m}^\ff$) by adding to $\Tt_{\vN,i}$ (resp. $\Tt_{\vN,c}$, $\Tt_{\vN,m}$), denoting by $\wt{x}$ the term (expressed here informally) $x/\max(\|x\|_\infty,1_\R)$, the following axiom:
$$\forall x^X \forall k^0 \exists y^X \preceq_X 1_X \left( \|\wt{x} -_X \tr(\wt{x}) \cdot_X 1_X\|_2 \leq_\R \|\wt{x}\cdot_X y-_X y\cdot_X\wt{x}\|_2 +\left(\frac1{k+1}\right)_\R \right),$$
while the theories for II$_1$ factors $\Tt_{\vN,i}^\fff$ (and also $\Tt_{\vN,c}^\fff$, $\Tt_{\vN,m}^\fff$) are constructed from $\Tt_{\vN,i}^\ff$ (resp. $\Tt_{\vN,c}^\ff$, $\Tt_{\vN,m}^\ff$) by adding the axiom
$$\forall k^0 \exists x^X \preceq_X 1_X \left( \|x \cdot_X *_X(x) -_X (x \cdot_X *_X(x))^2\|_2 + |\tr^{\Re}(x \cdot_X *_X(x)) - 1/\pi|_\R  \leq_\R \left(\frac1{k+1}\right)_\R \right).$$

For more case studies on adapting `naive' axioms to the form used in proof mining, see \cite[Section 7]{GunKoh16} and \cite[Section 3]{Sip19}.

\subsection{Models}\label{sec:models}

We shall now speak a bit about the semantics of these systems. To build a model of $\Tt_{\vN,i}$ (or, respectively, of $\Tt_{\vN,c}$ or $\Tt_{\vN,m}$), we proceed as follows. Firstly, given a set $A$, we define the  {\bf universe} associated to $A$ to be the family $\{\St_\rho\}_{\rho \in \T^X}$, built by setting $\St_0:=\N$ and $\St_X:=A$, and then setting, recursively, for each $\rho$, $\tau \in \T^X$, $\St_{\rho \to \tau}:=\St_\tau^{\St_\rho}$. Note that for types $\rho$ not involving abstract types, the set $\St_\rho$ does not depend on the choice of the set $X$ and we may refer to such sets even when we had not previously made such a choice. Then, to specify a {\bf set-theoretic model} (or simply a {\bf model}), one must first specify a set $A$ as before, and then, denoting the universe associated to it by $\{\St_\rho\}_{\rho \in \T^X}$, to specify, for every constant $c$ of type $\rho$ in the system (whether inherited or extended from $\mathsf{WE{\da}HA}^\omega$, or newly defined above), a corresponding interpretation $\St_c \in \St_\rho$. The semantics of formulas in such models is defined naturally in the Tarskian way and, of course, we do not count an assignment as above of sets and constants as a model except if it satisfies all the axioms and rules of the given system (including, e.g., for $\Tt_{\vN,m}$, the axioms of the form $c \lesssim_\rho m_c$).

The models which we shall mainly talk about interpret e.g. the constants $\|\cdot\|_2$ and $\|\cdot\|_\infty$ as somehow encoding norm functions, whereas they need to be interpreted as functions $X \to \N^\N$, so naturally the question of how to obtain a `canonical' representative for a non-negative real arises. This is done with the so-called `circle' function (first introduced in \cite{Koh05}), denoted by $(\cdot)_\circ$ and defined as follows \cite[p. 383]{Koh08}: for every $x \in \R_+$ and $n \in \N$, one sets
$$(x)_\circ(n):= j \left(2 \max\left\{ k \in \N\ \bigg|\ \frac{k}{2^{n+1}} \leq x \right\}, 2^{n+1} -1 \right).$$
This function has the following properties (cf. \cite[Lemma 17.8]{Koh08}):
\begin{itemize}
\item for each $x \in \R_+$, $(x)_\circ \in \N^\N$ is a monotone function which represents $x$ in the previous sense;
\item as restricted to natural numbers, the function is clearly primitive recursive (since the maximum may be taken simply as $2^{n+1} \cdot x$);
\item for each $x$, $y \in \R_+$ with $x \leq y$, we have that $(x)_\circ \leq_\R (y)_\circ$, $(x)_\circ \preceq_1 (y)_\circ$, and $(x)_\circ \lesssim_1 (y)_\circ$ (the relations being interpreted here semantically, as they do not depend on the choice of a particular set-theoretic model).
\end{itemize}

In order to be able to deal with the trace, we shall define another function $(\cdot)_+ : \R \to \N^\N$ which is able to represent all real numbers. For every $x \in \R$, if $x \geq 0$, then we set $(x)_+:=(x)_\circ$, otherwise, for every $n \in \N$, we set
$$(x)_+(n):= j \left(2 \max\left\{ k \in \N\ \bigg|\ \frac{k}{2^{n+1}} \leq -x \right\}-1, 2^{n+1} -1 \right).$$
One may then check that for each $x \in \R$, $(x)_+ \in \N^\N$ is a function which represents $x$ and that for each $x \in \R$ and $n \in \N$ with $|x| \leq n$, we have that $(x)_+ \preceq_1 (n)_+ = (n)_0$.

Let $A$ now be a tracial von Neumann algebra. We say now that the universe associated to it interprets $X$ as the set $A$, and that the set-theoretic model of $\Tt_{\vN,m}$ associated to it interprets the constants as follows:
\begin{itemize}
\item $0_X$, $+_X$, $-_X$, $\cdot$, $\cdot_X$, $1_X$, $*_X$ in the obvious way;
\item $m_{0_X}$ as the natural number $0$;
\item $m_{+_X}$ as natural number addition;
\item $m_{-_X}$ as the identity function on $\N$;
\item $m_{\cdot}$ as the mapping $f \mapsto (g \mapsto (x \mapsto (f(0)+g(0)+2) \cdot x ))$;
\item $m_{\cdot_X}$ as the mapping $n \mapsto (m \mapsto n \cdot m)$;
\item $m_{1_X}$ as the natural number $1$;
\item $m_{*_X}$ as the identity function on $\N$;
\item $\|\cdot\|_\infty$ as the mapping $x \mapsto (\|x\|_\infty)_\circ$ (as discussed previously);
\item $m_{\|\cdot\|_\infty}$ as the mapping $x \to (x)_\circ$ (which is clearly primitive recursive, as remarked previously);
\item $\|\cdot\|_2$ as the mapping $x \mapsto (\|x\|_2)_\circ$;
\item $m_{\|\cdot\|_2}$ as the mapping $x \mapsto (x)_\circ$;
\item $\tr$ (a pair of constants) as the mapping $x \mapsto ((\Re(\tr(x)))_+,(\Im(\tr(x)))_+)$;
\item $m_{\tr}$ (again a pair of constants) as the mapping $x \to ((x)_\circ,(x)_\circ)$;
\end{itemize}

\begin{theorem}\label{modelt}
Let $A$ be a tracial von Neumann algebra and let $\mathcal{S}$ be the set-theoretic model associated to it. Then $\mathcal{S} \models \Tt_{\vN,m}$.
\end{theorem}

\begin{proof}
What we need to show is that all the instantiations of the majorization constants are indeed majorants. For most of those, the arguments are largely similar to those for Banach spaces, and we refer to  \cite[p. 430]{Koh08} (the argument for $m_{\cdot}$ being a slight variation, due to the use of complex numbers). The only ones that raise some issues are the `tracial' ones, namely $m_{\|\cdot\|_2}$ and $m_{\tr}$.

For $m_{\|\cdot\|_2}$, the argument is simple. Let $x \in X$ and $n \in \N$ with $\|x\|_\infty \leq n$. We have to show that $(\|x\|_2)_\circ \lesssim_1 (n)_\circ$. Since $\|x\|_2 \leq \|x\|_\infty \leq n$, we immediately deduce the conclusion from the properties of $(\cdot)_\circ$.

For $m_{\tr}$, the argument is a bit more intricate, and we will only show it (w.l.o.g.) for the real part. Let $x \in X$ and $n \in \N$ with $\|x\|_\infty \leq n$. We have to show that $(\Re(\tr(x)))_+ \lesssim_1 (n)_\circ$, i.e. that for any $m$, $p \in \N$ with $m \leq p$, we have that $(\Re(\tr(x)))_+(m) \leq (n)_\circ(p)$. Let $m$, $p \in \N$ with $m \leq p$. Then, since
$$|\Re(\tr(x))| \leq |\tr(x)| \leq \|x\|_2 \leq \|x\|_\infty \leq n,$$
by the properties of $(\cdot)_+$ we have that $(\Re(\tr(x)))_+ \preceq_1 (n)_\circ$, so, using also that $(n)_\circ$ is monotone,
$$(\Re(\tr(x)))_+(m) \leq (n)_\circ(m)\leq (n)_\circ(p),$$
which is what we needed to show.

As remarked in \cite[p. 391]{Koh08}, the only additional, non-trivial issue to check is whether the quantifier-free rule of extensionality holds, and this is true because, in a tracial von Neumann algebra,  for all elements $x$, $y$ of it, $\|x-y\|_\infty = 0$ iff $x=y$.
\end{proof}

\begin{remark}[{cf. \cite[Definition 3.21]{Koh05}}]
We shall say, for any tracial von Neumann algebra $A$ and any sentence $B$ in our language, that $A \models B$ if the the set-theoretic model associated to it satisfies $B$. To simplify things, we may also use this notation when $B$ contains, in addition, elements of $A$, which are (of course) not in the base language.
\end{remark}

The following result is an immediate consequence of Theorem~\ref{modelt} and Proposition~\ref{factori}.

\begin{proposition}\label{factori2}
Let $A$ be a tracial von Neumann algebra and let $\mathcal{S}$ be the set-theoretic model associated to it. Then:
\begin{enumerate}[(i)]
\item $A$ is a tracial von Neumann factor iff $\mathcal{S} \models \Tt_{\vN,m}^\ff$.
\item $A$ is a II$_1$ factor iff $\mathcal{S} \models \Tt_{\vN,m}^\fff$.
\end{enumerate}
\end{proposition}

\section{Bound extraction}\label{sec:metath}

We say that a sentence is a {\bf $\Delta$-sentence} if it is of the form
$$\forall \uu{a}^{\uu{\delta}} \exists \uu{b}^{\uu{\sigma}} \preceq_{\uu{\sigma}} \uu{r} \uu{a} \forall \uu{c}^{\uu{\gamma}} B(\uu{a},\uu{b},\uu{c}),$$
where $B$ is quantifier-free and all its free variables are among the tuples $\uu{a}$, $\uu{b}$, $\uu{c}$, and $\uu{r}$ is a tuple of closed terms of type-tuple $\uu{\delta} \to \uu{\sigma}$ of $\Tt_{\vN,i}$, and we say that the {\bf Skolemization} of the sentence above is the sentence
$$\exists \uu{B}^{\uu{\delta} \to \uu{\sigma}} \preceq_{\uu{\delta} \to \uu{\sigma}} \uu{r} \forall \uu{a} \forall \uu{c} B(\uu{a},\uu{B}\uu{a},\uu{c}),$$
noting that if the sentence is purely universal (i.e. $\uu{\sigma}$ is the empty type-tuple), then it is its own Skolemization.

If $\Gamma$ is a set of $\Delta$-sentences, we denote by $\wt{\Gamma}$ the set of their Skolemizations. We note that if $\St$ is a set-theoretic model of $\Tt_{\vN,m}$ such that $\St\models\Gamma$, then $\St\models\wt{\Gamma}$.

We say that the Markov principle, denoted by $\mathsf{M}$, is the collection of all formulas of the form
$$(\neg\neg \exists \uu{x} B) \to \exists \uu{x} B,$$
where $B$ is a quantifier-free formula.

The main tool we shall use is an extension of Kohlenbach's monotone functional interpretation (see, e.g., \cite[Chapter 9]{Koh08}) to systems of abstract types. One first associates to every formula $B$ another formula $B_D$ in the calculus we have defined above, together with two disjoint term-tuples of variables $\uu{x}_B$ and $\uu{y}_B$, recursively, in the following way (with appropriate renamings of variables at crucial points):
\begin{itemize}
\item if $B$ is atomic, then $B_D$ is $B$ and $\uu{x}_B$ and $\uu{y}_B$ are empty;
\item if $B$ is of the form $C \land F$, then $B_D$ is $C_D \land F_D$, $\uu{x}_B$ is $\uu{x}_c$ concantenated with $\uu{x}_F$, while $\uu{y}_B$ is $\uu{y}_C$ concatenated with $\uu{y}_F$;
\item if $B$ is of the form $C \lor F$, then, taking $z$ to be a fresh variable of type $0$, $B_D$ is
$$(z=0 \to C_D) \land(\neg(z=0) \to F_D),$$
$\uu{x}_B$ is $(z)$ concatenated with $\uu{x}_C$ and $\uu{x}_F$, while $\uu{y}_B$ is $\uu{y}_C$ concatenated with $\uu{y}_F$;
\item if $B$ is of the form $C \to F$, then, taking $\uu{U}$ and $\uu{Y}$ to be tuples of fresh variables of suitable types, $B_D$ is
$$C_D[\uu{y}_C := \uu{Y} \uu{x}_C \uu{y}_F] \to F_D[\uu{x}_F := \uu{U} \uu{x}_C],$$
$\uu{x}_B$ is $\uu{U}$ concatenated with $\uu{Y}$, while $\uu{y}_B$ is $\uu{x}_C$ concatenated with $\uu{y}_F$;
\item if $B$ is of the form $\exists z C$, then $B_D$ is $C_D$, $\uu{x}_B$ is $(z)$ concatenated with $\uu{x}_C$, while $\uu{y}_B$ is $\uu{y}_C$;
\item if $B$ is of the form $\forall z C$, then, taking $\uu{X}$ to be a tuple of fresh variables of suitable types, $B_D$ is $C_D[\uu{x}_C := \uu{X}z]$, $\uu{x}_B$ is $\uu{X}$, while $\uu{y}_B$ is $(z)$ concatenated with $\uu{y}_C$.
\end{itemize}

\begin{theorem}\label{mono}
Let $\Gamma$ be a set of $\Delta$-sentences. Let $B$ be a formula and $\uu{a}$ be a term-tuple consisting of all the free variables of $B$. Assume that
$$\Tt_{\vN,i} + \mathsf{M} + \Gamma \vdash B.$$
Then one can extract from the proof a tuple $\uu{t}$ of closed terms of $\Tt_{\vN,m}$ with $\tp(\uu{t}) = \wh{\tp(\uu{a}) \to \tp(\uu{x}_B)}$, which are hereditarily of types not containing the type $X$, such that
$$\Tt_{\vN,m} + \wt{\Gamma} \vdash \exists \uu{x} ( \uu{x} \lesssim \uu{t} \land \forall\uu{a}\forall\uu{y}_BB_D[\uu{x}_B := \uu{x}\uu{a}]).$$
\end{theorem}

\begin{proof}
The proof extends \cite[Theorem 9.1]{Koh08} to all abstract finite types, using that all the axioms of $\Tt_{\vN,i}$ are purely universal and thus they are their own Skolemizations, so said Skolemizations are contained in $\Tt_{\vN,m}$. For constructing the corresponding majorants one uses the ideas in \cite[Proposition 6.6]{Koh08} and the newly added constants of $\Tt_{\vN,m}$.
\end{proof}

In order to get an interpretation which applies to the classical system $\Tt_{\vN,c}$, one must first apply a negative translation. The translation below is originally due to Kuroda \cite{Kur51}.

\begin{definition}[{\cite[Definition 10.1]{Koh08}}]
Let $B$ be a formula. We define $B^*$ recursively, as follows:
\begin{itemize}
\item if $B$ is atomic, then $B^*$ is $B$;
\item if $B$ is of the form $C \land F$, $C \lor F$ or $C \to F$, then $B^*$ is $C^* \land F^*$, $C^* \lor F^*$, resp. $C^* \to F^*$;
\item if $B$ is of the form $\exists x C$, then $B^*$ is $\exists x \left( C^* \right)$;
\item if $B$ is of the form $\forall x C$, then $B^*$ is $\forall x \neg\neg \left( C^* \right)$.
\end{itemize}
We then set $B^N := \neg\neg \left(B^*\right)$.
\end{definition}

The following result justifies the definition.

\begin{theorem}\label{neg}
Let $\Gamma$ be a set of $\Delta$-sentences. Let $B$ be a formula. Assume that
$$\Tt_{\vN,c} + \Gamma \vdash B.$$
Then
$$\Tt_{\vN,i} + \mathsf{M} + \Gamma \vdash B^N.$$
\end{theorem}

\begin{proof}
The proof extends \cite[Proposition 10.19]{Koh08}.
\end{proof}

We now present our syntactic metatheorem for tracial von Neumann algebras.

\begin{theorem}\label{comb}
Let $\Gamma$ be a set of $\Delta$-sentences. Let $B$ be a quantifier-free formula having at most the free variable $z$ of type $0$. Assume that
$$\Tt_{\vN,c}  + \Gamma \vdash \exists z B.$$
Then one can extract from the proof a closed term $t$ of $\Tt_{\vN,m}$ of type $0$, which is hereditarily of types not containing the type $X$, such that
$$\Tt_{\vN,m} + \wt{\Gamma} \vdash \exists z (z \leq t \land B).$$
\end{theorem}

\begin{proof}
By Theorem~\ref{neg}, it follows that
$$\Tt_{\vN,i} + \mathsf{M} + \Gamma \vdash \neg\neg  \exists z B.$$
Using $\mathsf{M}$, we get that
$$\Tt_{\vN,i} + \mathsf{M} + \Gamma \vdash  \exists z B.$$
Now, by Theorem~\ref{mono}, and using that $(\exists z B)_D = B$, $\uu{x}_{\exists z B} = (z)$, while $\uu{y}_{\exists z B}$ is empty, we get the desired conclusion.
\end{proof}

Its immediate semantic interpretation is as follows.

\begin{theorem}\label{semcor}
Let $\Gamma$ be a set of $\Delta$-sentences. Let $B$ be a quantifier-free formula having at most the free variable $z$ of type $0$. Assume that
$$\Tt_{\vN,c}  + \Gamma \vdash \exists z B.$$
Then one can extract from the proof a closed term $t$ of $\Tt_{\vN,m}$ of type $0$, which is hereditarily of types not containing the type $X$, such that for every set-theoretic model $\St$ of $\Tt_{\vN,m} + \Gamma$, we have that
$$\St \models \exists z (z \leq t \land B).$$
\end{theorem}

\begin{proof}
Considering the earlier remark stating that for every set-theoretic model $\St$ of $\Tt_{\vN,m}$ such that $\St\models\Gamma$, we have that $\St\models\wt{\Gamma}$, we just apply Theorem~\ref{comb} and we are done.
\end{proof}

We therefore obtain the following workable metatheorem for tracial von Neumann algebras.

\begin{theorem}\label{metric}
Let $\Gamma$ be a set of $\Delta$-sentences. Let $B$ be a quantifier-free formula having at most the free variable $z$ of type $0$. Assume that
$$\Tt_{\vN,c} + \Gamma \vdash \exists z B.$$
Then one can extract from the proof an $n \in \N$ such that for every tracial von Neumann algebra $A$ with $A \models \Gamma$, we have that
$$A \models \exists z (z \leq n \land B).$$
\end{theorem}

\begin{proof}
By Theorem~\ref{semcor}, we get that one can extract from the proof a closed term $t$ of $\Tt_{\vN,m}$ of type $0$, which is hereditarily of types not containing the type $X$, such that for every tracial von Neumann algebra $A$ with $A \models \Gamma$, we have that
$$A \models \exists z (z \leq t \land B).$$
We have that $t$ is interpreted uniformly in this class of models, and its interpretation gives the needed $n \in \N$.
\end{proof}

The following more general-looking variant is more in line with how metatheorems of proof mining are usually expressed.

\begin{theorem}\label{metric-cl}
Let $\Gamma$ be a set of $\Delta$-sentences. Let $B$ be a quantifier-free formula having at most the free variables $y$ of an arbitrary type $\tau$, and $z$ of type $0$. Assume that
$$\Tt_{\vN,c}  + \Gamma \vdash \forall y \exists z B.$$
Then one can extract from the proof a higher-order primitive recursive functional $\Phi : \St_{\wh{\tau}} \to \N$ such that for every tracial von Neumann algebra $A$ with $A \models \Gamma$, letting $\St$ be the model associated to $A$, and for every $b \in \St_\tau$ and $b' \in \St_{\wh{\tau}}$ such that we have that $\St \models b \lesssim_\tau b'$, we have that
$$A \models \exists z (z \leq \Phi(b') \land B[y:=b]).$$
\end{theorem}

\begin{proof}
Let $\Tt'_{\vN,i}$ be the system $\Tt_{\vN,i}$ with an additional constant $c_y$ of type $\tau$, and construct $\Tt'_{\vN,c}$ and $\Tt'_{\vN,m}$ in an analogous way (e.g. the last one containing a constant $m_{c_y}$ of type $\wh{\tau}$ with a corresponding axiom $c_y \lesssim_\rho m_{c_y}$). Then we have that
$$\Tt'_{\vN,c}  + \Gamma \vdash \exists z B[y:=c_y].$$
By the corresponding extension of Theorem~\ref{semcor} to these `primed' systems, we get that one can extract from the proof a closed term $t$ of $\Tt'_{\vN,m}$ of type $0$, which is hereditarily of types not containing the type $X$, such that for every set-theoretic model $\mathcal{U}$ of $\Tt'_{\vN,m} + \Gamma$, we have that
$$\mathcal{U} \models \exists z (z \leq t \land B[y:=c_y]).$$

Let $t'$ be the $\Tt_{\vN,m}$ term obtained by replacing every occurence of $m_{c_y}$ in $t$ by a fresh variable $w$ of type $\wh{\tau}$. We have that the closed $\Tt_{\vN,m}$ term $\lambda w. t'$ is interpreted uniformly in this class of models, and this interpretation  is the needed higher-order primitive recursive functional $\Phi : \St_{\wh{\tau}} \to \N$.

We now show that $\Phi$ has the desired properties. Let now $A$ be a tracial von Neumann algebra such that $A \models \Gamma$. Let $\St$ be the model associated to $A$ -- denoting its universe by  $\{\St_\rho\}_{\rho \in \T^X}$ -- and let $b \in \St_\tau$ and $b' \in \St_{\wh{\tau}}$ be such that $\St \models b \lesssim_\tau b'$. It is immediate that we can extend $\St$ to a model $\mathcal{U}$ of $\Tt'_{\vN,m} + \Gamma$ by interpreting $c_y$ as $b$ and $m_{c_y}$ as $b'$, and thus we get that
$$\mathcal{U} \models \exists z (z \leq t \land B[y:=c_y]),$$
i.e., that
$$\mathcal{U} \models \exists z (z \leq (\lambda w.t')(m_{c_y}) \land B[y:=c_y]),$$
from which we get that
$$\St \models \exists z (z \leq \Phi(b') \land B[y:=b]),$$
i.e. our conclusion.
\end{proof}

\begin{remark}
In light of Proposition~\ref{factori2}, we have that Theorem~\ref{metric-cl} remains true if we replace $\Tt_{\vN,c}$ by $\Tt_{\vN,c}^\ff$ (or $\Tt_{\vN,c}^\fff$) and `tracial von Neumann algebra' by `tracial von Neumann factor' (resp. `II$_1$ factor'), since all the additional axioms in those theories are $\Delta$-sentences and thus may be considered as part of the set $\Gamma$.
\end{remark}

\begin{remark}
The axiomatization given in \cite{FarHarShe14} and adapted here contains, as remarked earlier, only purely universal sentences, and it was indeed remarked in \cite[Section 5]{Gol12} that it should be so, since tracial von Neumann algebras are closed under substructures. Nevertheless, in the first preprint version \href{https://arxiv.org/abs/1004.0741v1}{arXiv:1004.0741v1} of \cite{FarHarShe14}, the authors give there an axiomatization whose last axiom would translate here (only) to a $\Delta$-sentence, and thus it would still make the proofs of the metatheorems go through, which is why we mention it here as an alternative. More specifically, one replaces the polynomial axiom schema with the following `polar decomposition' axiom (where $\wt{a}$ denotes the term expressed -- again, informally -- as $a/\max(\|a\|_\infty,(n)_\R)$):
$$\forall n^0 \forall a^X \exists b^X \preceq_X 1_X \exists c^X \preceq 1_X \exists d_X \preceq_X n \cdot 1_X \left( \|\wt{a}-bc\|_2 + \|c-d^*d\|_2 + \|b^*b-1_X\|_2 =_\R 0_\R \right),$$
which may be more convincing as it needs no elaborate preparations.
\end{remark}

\section{Conclusions and future work}\label{sec:conc}

The present paper represents a pilot study in doing proof mining on tracial von Neumann algebras. We have in mind some future possible applications: for example, from the result in \cite{ArzPau15}, which states that every finite list of permutations that `almost' commute is `near' another such list of actually commuting permutations and which is proven non-constructively using tracial von Neumann algebras, one could extract moduli quantifying exactly what is meant by `almost' and `near'. Another potential application would be a quantitative treatment of Walsh's ergodic theorem \cite{Wal12}, which was first stated (and proven) in terms of (real) $L^\infty$-algebras, but which has been restated in terms of abstract (real) tracial von Neumann algebras (`commutative probability spaces') by Tao in \cite{Tao12}.

As we mentioned in the Introduction, the manner in which the metatheorems were stated and proven here (using, as much as possible, a syntactic approach) is derived from the upcoming work \cite{CheLeuSipXX}, where an abstract metatheorem is proven which (hopefully) encompasses all the structures encountered so far in the proof mining literature.

In addition, some of the more theoretical possibilities of future research would be to investigate more other structures that need modifications of continuous logic (generalizing \cite[Section 6]{GunKoh16}; see also \cite{Cho16}), as well as the correspondence that arises in this context (again, as in \cite[Section 6]{GunKoh16}) between the uniform boundedness principle and model-theoretic ultraproducts, given that ultraproducts of tracial von Neumann algebras (and possibly other structures) have a somewhat different construction that the usual metric variant, using both the algebra norm $\|\cdot\|_\infty$ and the tracial norm $\|\cdot\|_2$ (for more details about this kind of `tracial' ultraproducts, see \cite{GeHad01, Gol12, FarHarShe14, Gol21, Gol22, GolHar}).

\section{Acknowledgements}

Liviu P\u aunescu was supported by a grant of the Romanian Ministry of Research, Innovation and Digitization, CNCS/CCCDI -- UEFISCDI, project number PN-III-P1-1.1-TE-2019-0262, within PNCDI III.

Andrei Sipo\c s was supported by a grant of the Romanian Ministry of Research, Innovation and Digitization, CNCS/CCCDI -- UEFISCDI, project number PN-III-P1-1.1-PD-2019-0396, within PNCDI III.

\end{document}